\documentclass[11pt]{amsart}  
\usepackage{amsthm, amsmath, amscd, amssymb, latexsym, stmaryrd, color}
\usepackage[top=3.0cm, bottom=3.0cm, left=3.0cm, right=3.0cm]{geometry}
\theoremstyle{plain}
\newtheorem*{thm}{Theorem \ref{main1}}
\newtheorem{theorem}{Theorem}[section]
\newtheorem{lemma}[theorem]{Lemma}
\newtheorem{proposition}[theorem]{Proposition}
\newtheorem{prop-def}[theorem]{Proposition-Definition}

\theoremstyle{definition}

\newtheorem{examples}[theorem]{Examples}
\newtheorem{remark}[theorem]{Remark}
\theoremstyle{remark}
\newtheorem*{ack}{Acknowledgement}
\usepackage[mathscr]{eucal}
\usepackage{graphics, graphpap}
\usepackage{array, tabularx, longtable}
\usepackage{color}
\numberwithin{equation}{section}

\def\s{\mathbf{s}}
\def\p{\mathrm{p}}

\def\red{\mathrm{red}}
\def\gcd{\mathrm{gcd}}
\def\mon{\mathrm{mon}}
\def\mot{\mathrm{mot}}
\def\ord{\mathrm{ord}}

\def\Sp{\mathrm{Sp}}
\def\l{\mathrm{left}}
\def\r{\mathrm{right}}
\def\Spec{\mathrm{Spec}}
\def\Lbb{\mathbb{L}}
\def\md{\mathrm{mod}\ }
\def\Var{\mathrm{Var}}

\def\sr{\mathrm{sr}}

\title{Motivic Milnor fibers of plane curve singularities}  
\author{L\^e Quy Thuong}
\address{Department of Mathematics, Vietnam National University \newline \indent 334 Nguyen Trai Street, Thanh Xuan District, Hanoi, Vietnam}
\email{leqthuong@gmail.com}
\thanks{This research is funded by Vietnam National Foundation for Science and Technology Development (NAFOSTED) under grant number FWO.101.2015.02.}


\begin{document}           
\begin{abstract}
We compute the motivic Milnor fiber of a complex plane curve singularity in an inductive and combinatoric way  using the extended simplified resolution graph. The method introduced in this article has a consequence that one can study the Hodge-Steenbrink spectrum of such a singularity in terms of that of a quasi-homogeneous singularity.
\end{abstract}
\maketitle                 

\section{Introduction}\label{sec1}
For more than two decades geometric motivic integration \cite{DL,DL1,DL2} has been a powerful tool in algebraic geometry and many branches of mathematics, in particular it has several important applications to singularity theory. Indeed, the work \cite{DL} of Denef-Loeser gives a breakthrough point of view in studying singularity theory, with the philosophy that motivic Milnor fiber is a motivic incarnation of the classical Milnor fiber. One shows that many invariants in singularity theory such as Hodge-Euler characteristic, Hodge polynomials and Hodge-Steenbrink spectra can be read off from motivic zeta functions and motivic Milnor fibers by means of appropriate Hodge realizations (cf. \cite{DL, DL1, DL2}, \cite{G}, \cite{GLM}). Motivic zeta functions and motivic Milnor fibers recently also bring considerable advances in the study of monodromy conjecture. Therefore, the computation of motivic zeta functions and motivic Milnor fibers is widely taken care by geometers and singularity theorists. For instance, Denef-Loeser \cite{DL1,DL4} describe explicitly the motivic zeta function and the motivic Milnor fiber of a regular function using resolution of singularity, Guibert \cite{G} and Steenbrink \cite{St} compute the motivic Milnor fiber of a non-degenerate singularity in terms of its Newton boundary. For composition, several aspects of the computation problem have been achieved due to \cite{DL3}, \cite{Loo}, \cite{GLM,GLM1,GLM2}, \cite{Thuong3}, etc.

In connection with the Hodge-Steenbrink spectrum the best result was quickly obtained by Denef-Loeser \cite{DL1}, which states that the spectrum of the motivic Milnor fiber and that of the classical Milnor fiber of a singularity are the same. Combing this with the motivic formula of Guibert-Loeser-Merle \cite{GLM} recovers a spectral formula conjectured in the 1980s by Steenbrink (proved earlier in \cite{Sa2}). Another interesting consequence of the result in \cite{DL1} is the proof by Budur \cite{Bu} on the equalities between the spectrum multiplicities and the inner jumping multiplicities with respect to the jumping numbers in $(0,1]\cap \mathbb Q$. In general, the motivic method in \cite{DL1} for computing spectra seems not to be easier than classical methods in practice. However, that the motivic method is more effective is possible provided the singularity is non-degenerate with respect to its Newton polyhedron or the singularity is irreducible plane curve as known. The latter is realized in \cite{G} and \cite{GVKM}, in which the spectrum of an irreducible plane curve singularity is computed via its motivic Milnor fiber, and this allows to reduce the problem to computing the spectrum of a quasi-homogeneous singularity of the form $(y^a+x^b)^N$, for some $a, b, N$ in $\mathbb N_{>0}$ and $(a,b)=1$.

In the present article we give a formula for the motivic Milnor fiber of an arbitrary complex plane curve singularity, which allows to express the spectrum of the singularity in terms of the spectra of quasi-homogeneous singularities of the form $\prod_{i=1}^r(\eta_iy^a+\xi_ix^b)^{N_i}$, for some $a, b, N_i$ in $\mathbb N_{>0}$ and $\xi_i, \eta_i$ in $\mathbb C$.  Let $f=(f,O)$ be the germ of a holomorphic function at the origin of $\mathbb{C}^2$. The classical Milnor fiber $M_f$ of $f$ is a crucial object for studying the singularity of $f$, it is up to homotopy nothing else than a bouquet of circles (cf. \cite{M}), the number of circles in the bouquet, i.e., the Milnor number of $f$, is a topological invariant. A generator of the fundamental group of a punctured disk, namely, the homotopic class of a small loop going once around the origin, which acts naturally on $H^i(M_f,\mathbb C)$, induces the monodromy $T_f$ of the singularity. To have a panorama picture about $f$ it requires at least profound knowledge from both $M_f$ and $T_f$ in the same importance. As explained in \cite{DL1}, the motivic Milnor fiber $\mathscr S_{f,O}$ of $f$, which lives in the Grothendieck ring $\mathscr M_{\mathbb C}^{\hat\mu}$ (cf. Section \ref{section3}), may carry information of both the Milnor fiber $M_f$ and the monodromy $T_f$. 

L\^e-Oka in \cite{LeOk} encode a resolution of singularity into a graph $\mathbf G$. Based on this, for $f$ plane curve singularity, our previous work \cite{Thuong2} constructs a so-called extended simplified resolution graph $\mathbf G_{\s}$ of $f$, which is independent of the choice of resolution for $f$, and using $\mathbf G_{\s}$ the monodromy zeta function $Z_{f,O}^{\mon}(t)$ can be described combinatorially and inductively (cf. \cite{Thuong2}). The present work is motivated by Denef-Loeser \cite{DL1} and Thuong \cite{Thuong2}, where we obtain a formula for $\mathscr S_{f,O}$ in terms of $\mathbf G_{\s}$. As in \cite{Thuong2} and Section \ref{sec2}, $\mathbf G_{\s}$ is arranged from simplified bamboos $\mathscr B$, in which each of the vertices $E(P_{\mathscr B,i})$ of $\mathbf G_{\s}$ in $\mathscr B$ (ordered naturally by $i$) attaches with a multiplicity $m(P_{\mathscr B,i})$ and a function of the form 
$$f_{P_{\mathscr B,i}}(u,v)=v^{\sum_{t=i+1}^{m_{\mathscr B}}a_{\mathscr B,t}A_{\mathscr B,t}}u^{m(P[\mathscr B])+\sum_{t=1}^{i-1}b_{\mathscr B,t}A_{\mathscr B,t}}\prod_{j=1}^{r_{\mathscr B,i}}(v^{a_{\mathscr B,i}}+\xi_{\mathscr B,i,j}u^{b_{\mathscr B,i}})^{A_{\mathscr B,i,j}},$$
while each edge connecting $E(P_{\mathscr B,i})$ and $E(P_{\mathscr B,i+1})$ attaches with the function 
$$f_{P_{\mathscr B,i,i+1}}(u,v)=v^{\sum_{t=i+1}^{m_{\mathscr B}}a_{\mathscr B,t}A_{\mathscr B,t}}u^{m(P[\mathscr B])+\sum_{t=1}^ib_{\mathscr B,t}A_{\mathscr B,t}},$$
where $A_{\mathscr B,i}=\sum_{j=1}^{r_{\mathscr B,i}}A_{\mathscr B,i,j}$, and $m(P[\mathscr B])$ is the multiplicity of $E(P[\mathscr B])$, the predecessor of $E(P_{\mathscr B,1})$ in $\mathbf G_{\s}$. 
The multiplicities are computed inductively via $\mathbf G_{\s}$ in \cite[Lemmas 3.1, 3.3]{Thuong2}. Put $\mu_n(\mathbb C)=\Spec(k[t]/(t^n-1))$, and
$$
X(\mathscr B,i):=\{(u,v)\in\mathbb A_{\mathbb C}^2\mid f_{P_{\mathscr B,i}}(u,v)=1\},$$
and 
$$
X(\mathscr B,i,i+1):=\{(u,v)\in\mathbb A_{\mathbb C}^2\mid f_{P_{\mathscr B,i,i+1}}(u,v)=1\}.$$

The following is the main result of the present article, which is proved in Section \ref{section4}.

\begin{thm}
The motivic Milnor fiber $\mathscr S_{f,O}$ is expressed as follows in the ring $\mathscr M_{\mathbb C}^{\hat\mu}$:
\begin{align*}
\mathscr S_{f,O}=\sum_{\mathscr B}\left(\sum_{i=1}^{m_{\mathscr B}}[X(\mathscr B,i)]-\sum_{i=1}^{m_{\mathscr B}-1}[X(\mathscr B,i,i+1)]-(\Lbb-1)\sum_{i=1}^{m_{\mathscr B}}\sum_{j=1}^{r_{\mathscr B,i}}[\mu_{A_{\mathscr B,i,j}}(\mathbb C)]\right),
\end{align*}
where the first sum $\sum_{\mathscr B}$ runs over the bamboos $\mathscr B$ of $\mathbf G_{\s}$.
\end{thm}


We aimed to treat a bigger problem, that proves the Hertling conjecture for an arbitrary complex plane curve singularity $f$. The conjecture for plane curve states that if $\mu$ is the Milnor number of $f$ and $\alpha_1\leq \dots\leq\alpha_{\mu}$ are spectral numbers of $f$ counted with the spectrum
multiplicities then $\frac{1}{\mu}\sum_{i=1}^{\mu}(\alpha_i-1)^2\leq \frac{1}{12}(\alpha_{\mu}-\alpha_1)$. Saito \cite{Saito} proves the conjecture for $f$ irreducible. To go further we need to compute the Hodge-Steenbrink spectrum $\Sp(f,O)$, or, modulo Theorem \ref{main1} and \cite{DL1}, to compute $\Sp(X(\mathscr B,i))$ and $\Sp(X(\mathscr B,i,i+1))$ (see Remark \ref{rk46}). However, it is well known that computing $\Sp(X(\mathscr B,i))$ and $\Sp(X(\mathscr B,i,i+1))$, as well as the spectrum of any quasi-homogeneous singularity, is still an open problem in general.


\section{Preliminaries on complex plane curve singularity, after \cite{Thuong2}}\label{sec2}
\subsection{Extended resolution graphs}\label{subsec2.1}
Let $f$ be the germ of a holomorphic function at the origin $O$ of $\mathbb{C}^2$, and let $C=f^{-1}(0)$. Remark that in the below arguments we need not to assume that $(f,O)$ is an isolated singularity. In \cite{LeOk}, L\^e-Oka define the extended resolution graph of a resolution of singularity, and their idea is the following. Let $\pi$ be a resolution of singularity of $f$ with $E_i$, $i\in A$, the irreducible components of $\pi^{-1}(C)_{\red}$. By \cite{LeOk}, the {\it extended resolution graph} $\mathbf G=\mathbf G(f,\pi)$ of $\pi$ is a graph whose vertices are $E_i$, $i\in A$, such that two vertices $E_i$ and $E_j$ are connected by a single edge if and only if the intersection $E_i\cap E_j$ is nonempty. There is a derived graph of $\mathbf G$ that will be useful in the later part of this paper, it is obtained from $\mathbf G$ by removing all the vertices of degree $2$, which is called {\it the simplified extended resolution graph} and denoted by $\mathbf G_{\s}$. 

In what follows we fix $\pi$ as a resolution of singularity of $f$ that is obtained by a set of toric modifications whose centers are determined canonically as in \cite{Thuong2}. Note that the graphs $\mathbf G$ and $\mathbf G_{\s}$ admit an explicit description, as in \cite{Thuong2}, and we shall use this description in the rest of the present paper. Let $\mathbf G_{\p}$ be the graph whose vertices correspond bijectively to the total space of each toric modification and the base space and whose edges correspond to each toric modification. Consider the origin $O$ of the base space as the root of $\mathbf G$ (and of $\mathbf G_{\s}$). We identify each face $P_i$ ($1\leq i\leq m$) of the Newton polyhedron $\Gamma$ of $f$ with a primitive weight vector and assume that $\det(P_i,P_{i+1})\geq 1$ for any $1\leq i\leq m-1$. Consider a regular simplicial cone subdivision $\{Q_1,\dots,Q_d\}$ which contains $P_1,\dots,P_m$. By definition, $Q_0=(1,0)^t$, $Q_{d+1}=(0,1)^t$, and the $Q_j$'s are primitive weight vectors and $\det(Q_j,Q_{j+1})=1$ for all $0\leq j\leq d$. Let $Q^{\l}:=Q_1$ and $Q^{\r}:=Q_d$. The first toric modification $\pi_O$ centered at the origin $O$ of $\mathbb{C}^2$ provides first vertices $E(Q_1),\dots,E(Q_d)$ of $\mathbf G$ and also provides first vertices $E(Q^{\l}), E(P_1),\dots,E(P_m), E(Q^{\r})$ of $\mathbf G_{\s}$. These give rise to a unique bamboo $\mathscr B_1$ (resp. $\mathscr B_{1,\s}$) in the first floor of $\mathbf G$ (resp. $\mathbf G_{\s}$). As explained in \cite{Thuong2}, we can take $Q^{\l}\not=P_1$ and $Q^{\r}\not=P_m$, and in view of (\ref{eq:2.3}) we can consider $\mathscr B_{1,\s}$ as independent of the choice of $\pi$.

Let $\pi_{\theta}:X_i\rightarrow X_j$ be an arbitrary toric modification different to $\pi_O$ which appears in $\mathbf G_{\p}$, where the center $\theta$ is one of the intersection points of $E(Q)$ and the strict transform of $C$ in $X_j$, with $Q$ a primitive weight vector of the previous toric modification $X_j\rightarrow X_k$. Assume by induction that the partial resolution graphs are already constructed and that $Q$ is the corresponding vertex of the simplified graph $\mathbf G_{\s}$. In \cite{AO}, A'Campo-Oka give a canonical way to determine a local coordinates $(u,v)$ at $\theta$ so that $u=0$ defines $E(Q)$, and in the coordinates one obtains the Newton polyhedron of the pullback of $f$. We denote the faces of this Newton polyhedron by $P_1^{\theta},\dots, P_{m_{\theta}}^{\theta}$, which may be ordered due to the condition $\det(P_i^{\theta},P_{i+1}^{\theta})\geq 1$ for all $1\leq i\leq m_{\theta}$. Fitting them into a regular simplicial cone subdivision $Q_1^{\theta},\dots, Q_{d_{\theta}}^{\theta}$, we get a toric modification at $\theta$. Then $E(Q_1^{\theta}),\dots, E(Q_{d_{\theta}}^{\theta})$ are the vertices in a horizontal bamboo of $\mathbf G$, say, $\mathscr B$, and $E(P_1^{\theta}),\dots, E(P_{m_{\theta}}^{\theta}),E(Q^{\theta,\r})$ are the vertices in a horizontal bamboo $\mathscr B_{\s}$ of $\mathbf G_{\s}$, where $Q^{\theta,\r}$ stands for $Q_{d_{\theta}}^{\theta}$. For convenience, we shall write from now on $Q_{\mathscr B,1},\dots, Q_{\mathscr B,d_{\mathscr B}}$ for $Q_1^{\theta},\dots, Q_{d_{\theta}}^{\theta}$ and write $P_{\mathscr B_{\s},1},\dots, P_{\mathscr B_{\s},m_{\mathscr B_{\s}}}, Q_{\mathscr B_{\s}}^{\r}$ for $P_1^{\theta},\dots, P_{m_{\theta}}^{\theta}, Q^{\theta,\r}$. We can assume $Q_{\mathscr B_{\s}}^{\r}\not=P_{\mathscr B_{\s},m_{\mathscr B_{\s}}}$, since if $P_{\mathscr B_{\s},m_{\mathscr B_{\s}}}$ has the form $(1,b)$ we may choose $Q_{\mathscr B_{\s}}^{\r}=(1,b+1)^t$. By (\ref{eq:2.8}) we may consider $\mathscr B_{\s}$ not to depend on $\pi_{\theta}$.  Now, we connect the vertex $E(Q_{\mathscr B,1})$ (resp. $E(P_{\mathscr B_{\s},1})$) in the present floor, say, $n$th, with the vertex $E(Q)$ in the $(n-1)$th floor by a non-horizontal single edge for $\mathbf G$ (resp. for $\mathbf G_{\s}$). For convenience, the connection edge is taken into account of the predecessor bamboo. Inductively, this process describes the extended resolution graph $\mathbf G$ and the simplified extended resolution graph $\mathbf G_{\s}$ (cf. \cite{Thuong2}).



\subsection{Multiplicities}\label{subsec2.2}
Let us write $f(x,y)$ as a finite product of irreducible germs, namely,
\begin{align}\label{eq:2.1}
f(x,y)=\prod_{i=1}^m\prod_{j=1}^{r_i}\prod_{l=1}^{s_{ij}}g_{i,j,l}(x,y),
\end{align}
where 
\begin{align}\label{eq:2.2}
g_{i,j,l}(x,y)=(y^{a_i}+\xi_{i,j}x^{b_i})^{A_{i,j,l}}+\text{(higher terms)}
\end{align}
are irreducible in $\mathbb{C}\{x\}[y]$ and $\xi_{i,j}$ are nonzero and distinct. Then the $\Gamma$-principal part of $f$, that corresponds to the first floor bamboo $\mathscr B_1$, is the following 
\begin{align}\label{GPP1}
f_{\mathscr B_{1,\s}}(x,y)=\prod_{i=1}^m\prod_{j=1}^{r_i}(y^{a_i}+\xi_{i,j}x^{b_i})^{A_{i,j}},
\end{align}
where 
$$A_{i,j}=\sum_{l=1}^{s_{ij}}A_{i,j,l}.$$ In fact, the $\Gamma$-principal part of $f$ only depends on $\mathscr B_{1,\s}$ and not on $\mathscr B_1$, hence the reason of notation $f_{\mathscr B_{1,\s}}$. The primitive weight vectors $P_i=(a_i,b_i)^t$, $1\leq i\leq m$, are ordered by the condition that $\det(P_i,P_{i+1})\geq 1$ for all $1\leq i< m$. The first toric modification $\pi_O$ of $f$ provides the unique bamboo at the first floor of $\mathbf G_{\s}$ whose vertices are $E(Q^{\l}), E(P_1),\dots, E(P_m)$ and $E(Q^{\r})$, in which the vertex $E(P_i)$ has degree $r_i+2$ for each $1\leq i\leq m$. Let $m(P)$ be the multiplicity of $\pi_O^*f$ on $E(P)$. By \cite[Lemma 3.1]{Thuong2}, we have
\begin{align}\label{eq:2.3}
m(Q^{\l})=\sum_{t=1}^ma_tA_t,\quad m(Q^{\r})=\sum_{t=1}^mb_tA_t
\end{align}
and
\begin{align}\label{eq:2.4}
m(P_i)=a_i\sum_{1\leq t\leq i}b_tA_t+b_i\sum_{i+1\leq t \leq m}a_tA_t,\ 1\leq i\leq m,
\end{align}
where $A_t=\sum_{j=1}^{r_t}\sum_{l=1}^{s_{t,j}}A_{t,j,l}$.


For a bamboo $\mathscr B\not=\mathscr B_1$ in $\mathbf G$, let $E(P)$ be the predecessor of $E(P_{\mathscr B_{\s},1})$ in $\mathbf G_{\s}$. Assume that the multiplicity $m(P)$ is already given. Let $\Phi$ be the resolution tower formed from the toric modifications earlier than the one corresponding to $\mathscr B_{\s}$ (with the center $\theta$ as above). Then, the pullback $\Phi^*f$ in the local coordinates $(u,v)$ at $\theta$ has the form
\begin{align}\label{eq:2.5}
\Phi^*f(u,v)=U(u,v)u^{m(P)}\prod_{i=1}^{m_{\mathscr B_{\s}}}\prod_{j=1}^{r_{\mathscr B_{\s},i}}\prod_{l=1}^{s_{\mathscr B_{\s},ij}}g_{\mathscr B_{\s},i,j,l}(u,v),
\end{align}
where 
\begin{align}\label{eq:2.6}
g_{\mathscr B_{\s},i,j,l}(u,v)=(v^{a_{\mathscr B_{\s},i}}+\xi_{\mathscr B_{\s},i,j}u^{b_{\mathscr B_{\s},i}})^{A_{\mathscr B_{\s},i,j,l}}+\text{(higher terms)}
\end{align}
are irreducible in $\mathbb{C}\{u\}[v]$, in which $\xi_{\mathscr B_{\s},i,j}$ are nonzero and distinct, and $U(u,v)$ is a unit in the ring $\mathbb{C}\{u,v\}$. Thus the $\Gamma(\Phi^*f)$-principal part of $\Phi^*f(u,v)$ is the following 
\begin{align}\label{GPPB}
f_{\mathscr B_{\s}}(u,v)=u^{m(P)}\prod_{i=1}^{m_{\mathscr B_{\s}}}\prod_{j=1}^{r_{\mathscr B_{\s},i}}(v^{a_{\mathscr B_{\s},i}}+\xi_{\mathscr B_{\s},i,j}u^{b_{\mathscr B_{\s},i}})^{A_{\mathscr B_{\s},i,j}},
\end{align}
where, by definition, 
$$A_{\mathscr B_{\s},i,j}=\sum_{l=1}^{s_{\mathscr B_{\s},ij}}A_{\mathscr B_{\s},i,j,l}.$$ 
Note that $f_{\mathscr B_{\s}}$ depends only on $\mathscr B_{\s}$ and not on $\mathscr B$. Furthermore, the primitive weight vectors $P_{\mathscr B_{\s},i}=(a_{\mathscr B_{\s},i},b_{\mathscr B_{\s},i})^t$, $1\leq i\leq m_{\mathscr B_{\s}}$, are ordered by the condition that $\det(P_{\mathscr B_{\s},i},P_{\mathscr B_{\s},i+1})\geq 1$ for all $1\leq i< m_{\mathscr B_{\s}}$. The corresponding toric modification gives rise to a bamboo of $\mathbf G_{\s}$ with vertices $E(Q_{\mathscr B_{\s}}^{\l}), E(P_{\mathscr B_{\s},1}),\dots, E(P_{\mathscr B_{\s},m})$ and $E(Q_{\mathscr B_{\s}}^{\r})$, and the degree of each $E(P_{\mathscr B_{\s},i})$ in $\mathbf G_{\s}$ is $r_{\mathscr B_{\s},i}+2$. By \cite[Lemma 3.3]{Thuong2}, we have, for $1\leq i\leq m_{\mathscr B_{\s}}$, 
\begin{equation}\label{eq:2.7}
\begin{aligned}
m(P_{\mathscr B_{\s},i})=a_{\mathscr B_{\s},i}m(P)&+a_{\mathscr B_{\s},i}\sum_{1\leq t\leq i}b_{\mathscr B_{\s},t}A_{\mathscr B_{\s},t}+b_{\mathscr B_{\s},i}\sum_{i+1\leq t \leq m_{\mathscr B_{\s}}}a_{\mathscr B_{\s},t}A_{\mathscr B_{\s},t}
\end{aligned}
\end{equation}
and
\begin{align}\label{eq:2.8}
m(Q_{\mathscr B_{\s}}^{\r})=m(P)+\sum_{t=1}^{m_{\mathscr B_{\s}}}b_{\mathscr B_{\s},t}A_{\mathscr B_{\s},t},
\end{align}
where $A_{\mathscr B_{\s},t}=\sum_{j=1}^{r_{\mathscr B_{\s},t}}\sum_{l=1}^{s_{\mathscr B_{\s},tj}}A_{\mathscr B_{\s},t,j,l}$ and $P$ is the predecessor of the first vertex of $\mathbf G_{\s}$ in $\mathscr B_{\s}$.

\section{The motivic Milnor fiber of a plane curve singularity}\label{section3}
\subsection{The Grothendieck ring of complex algebraic varieties with $\hat{\mu}$-action}
Consider the group schemes $\mu_n(\mathbb C)$ of $n$th roots of unity, the maps $\xi\mapsto \xi^k$, $n\geq 1$, $k\geq 1$, and let $\hat{\mu}$ be $\varprojlim\mu_n(\mathbb C)$. Let $\Var_{\mathbb C,\hat{\mu}}$ be the category of algebraic $\mathbb C$-varieties endowed with a $\hat{\mu}$-action. The Grothendieck group $K_0(\Var_{\mathbb C,\hat{\mu}})$ is an abelian group generated by symbols $[X]$ for $X$ in $\Var_{\mathbb C,\hat{\mu}}$ such that $[X]=[Y]$ whenever $X$ is $\hat{\mu}$-equivariant isomorphic to $Y$, $[X]=[Y]+[X\setminus Y]$ for $Y$ Zariski closed in $X$ with the $\hat{\mu}$-action induced from $X$ and $[X\times V]=[X\times\mathbb{A}_{\mathbb C}^e]$ if $V$ is an $e$-dimensional complex affine space with any linear $\hat{\mu}$-action and the action on $\mathbb{A}_{\mathbb C}^e$ is trivial. The abelian group $K_0(\Var_{\mathbb C,\hat{\mu}})$ becomes a ring with unit with respect to cartesian product. Let $\mathscr M_{\mathbb C}^{\hat{\mu}}$ denote $K_0(\Var_{\mathbb C,\hat{\mu}})[\Lbb^{-1}]$, where $\Lbb$ is $[\mathbb A_{\mathbb C}^1]$. 


\subsection{The motivic Milnor fiber of $(C,O)$}\label{subsec:3.2}
We continue studying the plane curve singularity $(C,O)$ in Section \ref{sec2}. In the sequel, we shall express the motivic Milnor fiber $\mathscr S_{f,O}$ in terms of the extended simplified resolution graph $\mathbf G_{\s}$. 

For $n\geq 1$, we define a $\mathbb C$-variety 
\begin{align}\label{eq:3.1}
J_n=J_n(f):=\left\{(\varphi,\psi)\in \left(t\mathbb C[t]/t^{n+1}\right)^2 \mid f(\varphi,\psi)=t^n\md  t^{n+1}\right\},
\end{align}
endowed with a $\hat{\mu}$-action via $\mu_n(\mathbb C)$ defined by $\xi\cdot(\varphi(t),\psi(t))=(\varphi(\xi t),\psi(\xi t))$, $\xi$ in $\mu_n(\mathbb C)$. Thus, it gives rise to an element $[J_n]$ in $\mathscr M_{\mathbb C}^{\hat{\mu}}$. Consider the formal power series
\begin{align}\label{eq:3.2}
Z_{f,O}^{\mot}(T):=\sum_{n\geq 1}[J_n]\Lbb^{-2n}T^n
\end{align}
in $\mathscr M_{\mathbb C}^{\hat{\mu}}[[T]]$, which is called the {\it motivic zeta function} of the singularity $f$. It was proved in \cite{DL} that $Z_{f,O}^{\mot}(T)$ is an element of $\mathscr M_{\mathbb C}^{\hat{\mu}}[[T]]_{\sr}$, the $\mathscr M_{\mathbb C}^{\hat{\mu}}$-submodule of $\mathscr M_{\mathbb C}^{\hat{\mu}}[[T]]$ generated by 1 and by finite products of $\frac{\Lbb^pT^q}{1-\Lbb^pT^q}$ with $(p,q)$ in $\mathbb Z\times\mathbb N_{>0}$. Moreover, there exists a unique $\mathscr M_{\mathbb C}^{\hat{\mu}}$-linear morphism 
$$\lim_{T\rightarrow\infty}: \mathscr M_{\mathbb C}^{\hat{\mu}}[[T]]_{\sr}\to \mathscr M_{\mathbb C}^{\hat{\mu}}$$
such that $\lim_{T\rightarrow\infty}\frac{\Lbb^pT^q}{1-\Lbb^pT^q}=-1$. Then one defines the {\it motivic Milnor fiber} of the singularity $f$ to be $-\lim_{T\rightarrow\infty}Z_{f,O}^{\mot}(T)$ and denotes it by $\mathscr S_{f,O}$.


\subsection{Denef-Loeser's formula}
%
Consider the resolution of singularity $\pi$ of $(C,O)$ in Subsection \ref{subsec2.1}. Putting $E_i^{\circ}=E_i\setminus\bigcup_{j\not=i}E_j$ with $i$, $j$ in $A$, locally, on $E_i^{\circ}$, $f(\pi(u,v))= \widetilde U(u,v)u^{m_i}$, and on $E_{i,j}^{\circ}=E_i\cap E_j\not=\emptyset$, $f(\pi(u,v))= \widetilde U(u,v)u^{m_i}v^{m_j}$, where $\widetilde U(u,v)$ is a unit. The Denef-Loeser unramified Galois covering 
$$\pi_S:\widetilde{E}_S^{\circ}\to E_S^{\circ}$$ 
with Galois group $\mu_{m_S(\mathbb C)}$ is defined locally by  
$$\{(u,v;z)\in E_S^{\circ}\times \mathbb{A}_{\mathbb C}^1 \mid z^{m_S}=\widetilde{U}(u,v)^{-1}\},$$ 
which is endowed with the $\mu_{m_S}(\mathbb C)$-action induced by multiplying the $z$-coordinate with the $m_S$-roots of unity, where $S$ is $\{i\}$ or $\{i,j\}$, $m_{i,j}$ is the greatest common divisor of $m_i$ and $m_j$. 

\begin{theorem}[Denef-Loeser, \cite{DL, DL2}]\label{DLSf}
With the previous notation, the identity
\begin{align*}
\mathscr S_{f,O}=\sum_{i\in A}[\widetilde E_i^{\circ}\cap\pi^{-1}(O)]-(\Lbb-1)\sum_{i,j\in A}[\widetilde E_{i,j}^{\circ}\cap\pi^{-1}(O)]
\end{align*}
holds in $\mathscr M_{\mathbb C}^{\hat{\mu}}$.
\end{theorem}

If $E_i$ is contained in $\pi^{-1}(O)$, we have $[\widetilde E_i^{\circ}\cap\pi^{-1}(O)]=[\widetilde E_i^{\circ}]$; otherwise, $[\widetilde E_i^{\circ}\cap\pi^{-1}(O)]=0$. If $E_{i,j}^{\circ}\not=\emptyset$, then $[\widetilde E_{i,j}^{\circ}\cap\pi^{-1}(O)]=[\mu_{m_{i,j}}(\mathbb C)]$. One thus deduces from Theorem \ref{DLSf} that
\begin{align}\label{DL}
\mathscr S_{f,O}=\sum_{E_i \text{ exceptional}}[\widetilde E_i^{\circ}]-(\Lbb-1)\sum_{i,j}[\mu_{m_{i,j}}(\mathbb C)].
\end{align}

The following is a consequence of Theorem \ref{DLSf}, via (\ref{DL}), together with using $\mathbf G_{\s}$.

\begin{proposition}\label{prop:3.2} 
In $\mathscr M_{\mathbb C}^{\hat{\mu}}$, one has
\begin{align}\label{eq:3.3}
\mathscr S_{f,O}=\mathscr S_{f_{\mathscr B_{1,\s}},O}+\sum_{\mathscr B\in \mathbf B, \mathscr B\not=\mathscr B_1}&\Big(\mathscr S_{f_{\mathscr B_{\s}},O}+(\Lbb-1)[\mu_{\gcd(m(P[\mathscr B_{\s}]),\sum_{t=1}^{m_{\mathscr B_{\s}}}a_{\mathscr B_{\s},t}A_{\mathscr B_{\s},t})}(\mathbb C)]\Big),
\end{align}
where $m(P[\mathscr B_{\s}])$ is the multiplicity of $E(P[\mathscr B_{\s}])$, the predecessor of $E(P_{\mathscr B_{\s},1})$ in $\mathbf G_{\s}$. 
\end{proposition}

\begin{remark}
The first bamboo is also the top bamboo if and only if $f(x,y)=x^A$ for some $A$ in $\mathbb N_{>0}$; in this case, $\mathscr S_{f,O}=[\mu_A(\mathbb C)]$. If $\mathscr B\not=\mathscr B_1$ is a top bamboo, the term corresponding to $\mathscr B$ in the sum (\ref{eq:3.3}) vanishes in $\mathscr M_{\mathbb C}^{\hat{\mu}}$.
\end{remark}

\begin{proof}[Proof of Proposition \ref{prop:3.2}]
For a bamboo $\mathscr B$ in $\mathbf G$ we consider the principal part $f_{\mathscr B_{\s}}$ given in (\ref{GPP1}) and (\ref{GPPB}). Let $C_{\mathscr B_{\s},i,j}$ be the strict transform of the germ (not necessarily reduced) $(\{f_{\mathscr B_{\s}}=0\},O)$, in the toric modification corresponding to $\mathscr B$, which intersects with $E(P_{\mathscr B_{\s},i})$, $1\leq i\leq m_{\mathscr B_{\s}}$, $1\leq j\leq r_{\mathscr B_{\s},i}$. Let $\mathscr B^{i,j}$ be the successor of $\mathscr B$ in the set $\mathbf B$ of bamboos of $\mathbf G$ whose indices $(i,j)$ correspond to those of $C_{\mathscr B_{\s},i,j}$. We remark that, if $\mathscr B$ is not a top bamboo, then $m(C_{\mathscr B_{\s},i,j})=A_{\mathscr B_{\s},i,j}$ and 
\begin{equation}\label{rmk}
\gcd(m(E(P_{\mathscr B_{\s},i,})),A_{\mathscr B_{\s},i,j})=\gcd(m(E(P_{\mathscr B_{\s},i,})),m(E(Q_{\mathscr B_{\s}^{i,j},1}))). 
\end{equation}
The latter comes from the fact that there exists a natural number $N_{\mathscr B}$ such that 
$$m(E(Q_{\mathscr B_{\s}^{i,j},1}))=N_{\mathscr B}m(E(P_{\mathscr B_{\s},i,}))+A_{\mathscr B_{\s},i,j}.$$ 
For convenience, we define $[\mu_{\gcd(m(E(P_{\mathscr B_{\s},i})),m(E(Q_{\mathscr B_{\s}^{i,j},1})))}(\mathbb C)]=0$ if $\mathscr B$ a top bamboo. 

We now apply Theorem \ref{DLSf} to $f_{\mathscr B_{\s}}$, with (\ref{rmk}) used. If $\mathscr B=\mathscr B_1$, 
\begin{equation}\label{eq:3.5}
\begin{aligned}
\mathscr S_{f_{\mathscr B_{1,\s}},O}=&\sum_{j=1}^d[\widetilde E(Q_j)^{\circ}]-(\Lbb-1)\sum_{j=1}^{d-1}[\mu_{\gcd(m(E(Q_j)),m(E(Q_{j+1}))}(\mathbb C)]\\
&-(\Lbb-1)\sum_{i=1}^m\sum_{j=1}^{r_i}[\mu_{\gcd(m(E(P_i)),m(E(Q_{\mathscr B_{1,\s}^{i,j},1})))}(\mathbb C)].
\end{aligned}
\end{equation}
Also, if $\mathscr B\not=\mathscr B_1$,
\begin{equation}\label{eq:3.6}
\begin{aligned}
\mathscr S_{f_{\mathscr B_{\s}},O}=&\sum_{j=1}^{d_{\mathscr B}}[\widetilde E(Q_{\mathscr B,j})^{\circ}]-(\Lbb-1)[\mu_{\gcd(m(P[\mathscr B_{\s}]),\sum_{t=1}^{m_{\mathscr B_{\s}}}a_{\mathscr B_{\s},t}A_{\mathscr B_{\s},t})}(\mathbb C)]\\
& -(\Lbb-1)\sum_{j=1}^{d_{\mathscr B}-1}[\mu_{\gcd(m(E(Q_{\mathscr B,j})),m(E(Q_{\mathscr B,j+1}))}(\mathbb C)]\\
& -(\Lbb-1)\sum_{i=1}^{m_{\mathscr B_{\s}}}\sum_{j=1}^{r_{\mathscr B_{\s},i}}[\mu_{\gcd(m(E(P_{\mathscr B_{\s},i,})),m(E(Q_{\mathscr B_{\s}^{i,j},1})))}(\mathbb C)].
\end{aligned}
\end{equation}
Thanks to \cite{DL,DL2}, in any case, $\mathscr S_{f_{\mathscr B_{\s}},O}$ is independent of the choice of the resolution of singularity, i.e., of $\mathscr B$. Now, using (\ref{eq:3.5}), (\ref{eq:3.6}) and (\ref{rmk}), the right hand side of (\ref{eq:3.3}) equals
\begin{equation*}
\begin{gathered}
\sum_{\mathscr B\in \mathbf B}\Big(\sum_{j=1}^{d_{\mathscr B}}[\widetilde E(Q_{\mathscr B,j})^{\circ}]-(\Lbb-1)\sum_{j=1}^{d_{\mathscr B}-1}[\mu_{\gcd(m(E(Q_{\mathscr B,j})),m(E(Q_{\mathscr B,j+1}))}(\mathbb C)]\\
-(\Lbb-1)\sum_{i=1}^{m_{\mathscr B_{\s}}}\sum_{j=1}^{r_{\mathscr B_{\s},i}}[\mu_{\gcd(m(E(P_{\mathscr B_{\s},i})),m(E(Q_{\mathscr B_{\s}^{i,j},1})))}(\mathbb C)]\Big)
\end{gathered}
\end{equation*}
which is nothing but $\mathscr S_{f,O}$, again by Theorem \ref{DLSf}.
\end{proof}

\section{The main result}\label{section4}

\subsection{Direct computations}\label{subsec:3.4}
In this paragraph, a direct computation of $\mathscr S_{f_{\mathscr B_{\s}}}$ for a bamboo $\mathscr B$ in $\mathbf G$ will be given. We define $m(P[\mathscr B_{1,\s}])=0$. For any bamboo $\mathscr B$, the face functions $f_{P_{\mathscr B_{\s},i}}(u,v)$ of the Newton polyhedron $\Gamma(f_{\mathscr B_{\s}})$ are equal to
$$\left(\prod_{t=1}^{i-1}\prod_{j=1}^{r_{\mathscr B,t}}\xi_{\mathscr B,t,j}^{A_{\mathscr B,t,j}}\right)v^{\sum_{t=i+1}^{m_{\mathscr B_{\s}}}a_{\mathscr B_{\s},t}A_{\mathscr B_{\s},t}}u^{m(P[\mathscr B_{\s}])+\sum_{t=1}^{i-1}b_{\mathscr B_{\s},t}A_{\mathscr B_{\s},t}}$$
times
$$\prod_{j=1}^{r_{\mathscr B_{\s},i}}(v^{a_{\mathscr B_{\s},i}}+\xi_{\mathscr B_{\s},i,j}u^{b_{\mathscr B_{\s},i}})^{A_{\mathscr B_{\s},i,j}}$$
for $1\leq i\leq m_{\mathscr B_{\s}}$, with convention $\sum_{t=1}^0=0$ and $\sum_{t=m_{\mathscr B_{\s}}+1}^{m_{\mathscr B_{\s}}}=0$. Let $P_{\mathscr B_{\s},i,i+1}$ be the intersection of two faces $P_{\mathscr B_{\s},i}$ and $P_{\mathscr B_{\s},i+1}$ for $0\leq i\leq m_{\mathscr B}$ (here $P_{\mathscr B_{\s},0,1}$ denotes the left end point of the face $P_{\mathscr B_{\s},1}$ and $P_{\mathscr B_{\s},m,m+1}$ denotes the right end point of the face $P_{\mathscr B_{\s},m}$). Then, using the previous convention,
\begin{align*}
f_{P_{\mathscr B_{\s},i,i+1}}(u,v)=\left(\prod_{t=1}^{i-1}\prod_{j=1}^{r_{\mathscr B,t}}\xi_{\mathscr B,t,j}^{A_{\mathscr B,t,j}}\right)v^{\sum_{t=i+1}^{m_{\mathscr B_{\s}}}a_{\mathscr B_{\s},t}A_{\mathscr B_{\s},t}}u^{m(P[\mathscr B_{\s}])+\sum_{t=1}^ib_{\mathscr B_{\s},t}A_{\mathscr B_{\s},t}}.
\end{align*}

In what follows, by $g^{-1}(1)$ we shall mean the variety $\{(u,v)\in\mathbb G_{m,\mathbb C}^2\mid g(u,v)=1\}$ for a polynomial $g$ in $\mathbb C[u,v]$.

\begin{proposition}\label{prop:3.3} 
For $\mathscr B=\mathscr B_1$,
\begin{align*}
\mathscr S_{f_{\mathscr B_{1,\s}},O}=&\sum_{i=1}^m[f_{P_i}^{-1}(1)]-\sum_{i=1}^{m-1}[f_{P_{i,i+1}}^{-1}(1)]-(\Lbb-1)\sum_{i=1}^m\sum_{j=1}^{r_i}[\mu_{A_{i,j}}(\mathbb C)]\\
&+[\mu_{m(Q^{\l})}(\mathbb C)]+[\mu_{m(Q^{\r})}(\mathbb C)].
\end{align*}
For $\mathscr B\not=\mathscr B_1$,
\begin{align*}
\mathscr S_{f_{\mathscr B_{\s}},O}=&\sum_{i=1}^{m_{\mathscr B_{\s}}}[f_{P_{\mathscr B_{\s},i}}^{-1}(1)]-\sum_{i=1}^{m_{\mathscr B_{\s}}-1}[f_{P_{\mathscr B_{\s},i,i+1}}^{-1}(1)]-(\Lbb-1)\sum_{i=1}^{m_{\mathscr B_{\s}}}\sum_{j=1}^{r_{\mathscr B_{\s},i}}[\mu_{A_{\mathscr B_{\s},i,j}}(\mathbb C)]\\
&-(\Lbb-1)[\mu_{\gcd(m(P[\mathscr B_{\s}]),\sum_{t=1}^{m_{\mathscr B_{\s}}}a_{\mathscr B_{\s},t}A_{\mathscr B_{\s},t})}(\mathbb C)]+[\mu_{m(Q_{\mathscr B_{\s}}^{\r})}(\mathbb C)].
\end{align*}
\end{proposition}

\begin{proof}
Define a function $\ell_{\mathscr B_{\s}}:\mathbb R_{\geq 0}^2\to \mathbb R_{\geq 0}$ by setting $\ell(p,q)=\min_{(a,b)\in\Gamma_{\mathscr B_{\s}}}(ap+bq)$, where $\Gamma_{\mathscr B_{\s}}$ denotes $\Gamma(f_{\mathscr B_{\s}})$, for each bamboo $\mathscr B$ in $\mathbf G$. For $\Omega=(p,q)$ in $\mathbb{R}_{\geq 0}^2$, let $\gamma_{\Omega}$ be the face consisting of points $(a,b)$ in $\Gamma_{\mathscr B_{\s}}$ with $ap+bq=\ell_{\mathscr B_{\s}}(\Omega)$. For a face $P$, let $\sigma(P)$ be the sets of points $\Omega$ of $\mathbb{R}_{\geq 0}^2$ with $\gamma_{\Omega}=P$; thus the closure $\overline{\sigma}(P)$ is the sets of points $\Omega$ in $\mathbb{R}_{\geq 0}^2$ with $\gamma_{\Omega}\supset P$. When $P$ runs over the faces of $\Gamma_{\mathscr B_{\s}}$, $\overline{\sigma}(P)$ form a fan in $\mathbb{R}_{\geq 0}^2$ partitioning it into rational polyhedral cones (see \cite{Thuong1}).

Let $J_{\mathscr B_{\s},n}$ be the scheme defined similarly as $J_n$ of (\ref{eq:3.1}) with $f$ replaced by $f_{\mathscr B_{\s}}$, that is, $J_{\mathscr B_{\s},n}:=J_n(f_{\mathscr B_{\s}})$. For $(p,q)$ in $\mathbb{N}_{>0}^2$, let $J_{\mathscr B_{\s},n}(p,q)$ be the set of $(\varphi,\psi)$ in $J_{\mathscr B_{\s},n}$ such that there exists $(\varphi',\psi')$ in $\mathbb C[[t]]^2$ of order $(p,q)$ with $\varphi=\varphi' \md t^{n+1}$ and $\psi=\psi' \md t^{n+1}$. It gives rise in a natural way to an element $[J_{\mathscr B_{\s},n}(p,q)]$ of $\mathscr M_{\mathbb C}^{\hat{\mu}}$. Since $J_{\mathscr B_{\s},n}$ is a disjoint union of such these sets and $n\geq \ell_{\mathscr B_{\s}}(\Omega)$ for any $\Omega$ in $\mathbb N_{>0}^2$, we have
\begin{align}\label{eq:3.9}
Z_{f_{\mathscr B_{\s}},O}^{\mot}(T)=\sum_{\Omega\in\mathbb{N}_{>0}^2}(\Phi_1(T)+\Phi_2(T)),
\end{align}
where 
\begin{equation*}
\begin{cases}
\Phi_1(T)=[J_{\mathscr B_{\s},\ell_{\mathscr B_{\s}}(\Omega)}(\Omega)]\Lbb^{-2\ell_{\mathscr B_{\s}}(\Omega)}T^{\ell_{\mathscr B_{\s}}(\Omega)}\\
\Phi_2(T)=\sum_{n>\ell_{\mathscr B_{\s}}(\Omega)}[J_{\mathscr B_{\s},n}(\Omega)]\Lbb^{-2n}T^n.
\end{cases}
\end{equation*}
Now, by writing (\ref{eq:3.9}) as  
\begin{align*}
\sum_{\Omega\in\mathbb{N}^2}=\sum_{i=1}^{m_{\mathscr B_{\s}}}\sum_{\Omega\in\sigma(P_{\mathscr B_{\s},i})}+\sum_{i=0}^{m_{\mathscr B_{\s}}}\sum_{\Omega\in\sigma(P_{\mathscr B_{\s},i,i+1})}.
\end{align*}
and by using Lemmas \ref{lem:3.4} and \ref{lem:3.5}, we complete the proof of the proposition.
\end{proof}

\begin{lemma}\label{lem:3.4} 
The following identities hold:
\begin{align*}
\lim_{T\to\infty}\sum_{\Omega\in\sigma(P_{\mathscr B_{\s},i})}\Phi_1(T)&=-[f_{P_{\mathscr B_{\s},i}}^{-1}(1)], \quad 1\leq i\leq m_{\mathscr B_{s}},\\
\lim_{T\to\infty}\sum_{\Omega\in\sigma(P_{\mathscr B_{\s},i,i+1})}\Phi_1(T)&=
\begin{cases}
-[\mu_{m(Q^{\l})}(\mathbb C)], &\mathscr B_{\s}=\mathscr B_{1,\s}, i=0,\\
[f_{P_{\mathscr B_{\s},0,1}}^{-1}(1)], &\mathscr B_{\s}\not=\mathscr B_{1,\s}, i=0,\\
[f_{P_{\mathscr B_{\s},i,i+1}}^{-1}(1)], &1\leq i< m_{\mathscr B_{s}},\\
-[\mu_{m(Q_{\mathscr B_{\s}}^{\r})}(\mathbb C)], &i=m_{\mathscr B_{s}}.
\end{cases}
\end{align*}
In particular, for $\mathscr B_{\s}\not=\mathscr B_{1,\s}$, 
\begin{align*}
[f_{P_{\mathscr B_{\s},0,1}}^{-1}(1)]&=(\Lbb-1)[\mu_{\gcd(m(P[\mathscr B_{\s}]),\sum_{t=1}^{m_{\mathscr B_{\s}}}a_{\mathscr B_{\s},t}A_{\mathscr B_{\s},t})}(\mathbb C)].
\end{align*}
\end{lemma}

\begin{proof}
Note that an element of $\sigma(P_{\mathscr B_{\s},i})$ has the form $\Omega=(p,q)=(\alpha a_{\mathscr B_{\s},i},\alpha b_{\mathscr B_{\s},i})$, $\alpha>0$, and for such an $\Omega$, $\ell_{\mathscr B_{\s}}(\Omega)=\alpha m(P_{\mathscr B_{\s},i})$. Thus $p<\ell_{\mathscr B_{\s}}(p,q)$ and $q<\ell_{\mathscr B_{\s}}(p,q)$ for any $(p,q)$ in $\mathbb N_{>0}^2$. For $(\varphi(t),\psi(t))$ in $J_{\mathscr B_{\s},\ell_{\mathscr B_{\s}}(\Omega)}(\Omega)$, we have 
$$f_{\mathscr B_{\s}}(\varphi(t),\psi(t))=f_{P_{\mathscr B_{\s},i}}(c_p,d_q)t^{\ell_{\mathscr B_{\s}}(\Omega)}+(\text{higher terms}),$$ 
with $(c_p,d_q)$ in $\mathbb G_{\mathbb C}^2$ the coefficients of $(t^p,t^q)$ in $(\varphi(t),\psi(t))$. It turns out that, in $\mathscr M_{\mathbb C}^{\hat{\mu}}$,
$$[J_{\mathscr B_{\s},\ell_{\mathscr B_{\s}}(\Omega)}(\Omega)]=[f_{P_{\mathscr B_{\s},i}}^{-1}(1)]\Lbb^{2\ell_{\mathscr B_{\s}}(\Omega)-|\Omega|},$$
where $|\Omega|=p+q$. By \cite[Lemma 2.1.5]{G},
\begin{align*}
\lim_{T\to\infty}\sum_{\Omega\in\sigma(P_{\mathscr B_{\s},i})}\Phi_1(T)=[f_{P_{\mathscr B_{\s},i}}^{-1}(1)]\lim_{T\to\infty}\sum_{\Omega\in\sigma(P_{\mathscr B_{\s},i})}\Lbb^{-|\Omega|}T^{\ell_{\mathscr B_{\s}}(\Omega)}=-[f_{P_{\mathscr B_{\s},i}}^{-1}(1)],
\end{align*}
the first identity is proved.

For $1\leq i< m_{\mathscr B_{s}}$, the cone $\sigma(P_{\mathscr B_{\s},i,i+1})$ consists of $\Omega=(\alpha a_{\mathscr B_{\s},i}+\beta a_{i+1},\alpha b_{\mathscr B_{\s},i}+\beta b_{\mathscr B_{\s},i+1})$ for $\alpha$ and $\beta$ in $\mathbb N_{>0}$, and 
$$\ell_{\mathscr B_{\s}}(\Omega)=\alpha m(P_{\mathscr B_{\s},i})+\beta m(P_{\mathscr B_{\s},i+1}).$$ 
Meanwhile, for $\mathscr B_{\s}\not=\mathscr B_{1,\s}$, the cone $\sigma(P_{\mathscr B_{\s},0,1})$ consists of $\Omega=(\alpha+\beta a_{\mathscr B_{\s},1},\beta b_{\mathscr B_{\s},1})$ for $\alpha$ and $\beta$ in $\mathbb N_{>0}$, and 
$$\ell_{\mathscr B_{\s}}(\Omega)=\alpha m(P[\mathscr B_{\s}])+\beta m(P_{\mathscr B_{\s},1}).$$ 
Then the results follow by using the same arguments as for the first identity. 

Now, we consider the case where $\mathscr B_{\s}=\mathscr B_{1,\s}$ and $i=0$, and rewrite for simplicity that $\ell=\ell_{\mathscr B_{1,\s}}$, $J_n^1=J_{\mathscr B_{1,\s},n}$. Then, for any $\Omega=(p,q)=(\alpha+\beta a_1,\beta b_1)$ in $\sigma(P_{0,1})$ ($\alpha,\beta$ in $\mathbb N_{>0}$), we have $\ell(\Omega)=qm(Q^{\l})$. If $p>qm(Q^{\l})$, an element of $J_{\ell(\Omega)}(\Omega)$ is of the form 
$$(0,\psi)=(0,d_qt^q+\cdots+d_{qm(Q^{\l})}t^{qm(Q^{\l})})$$
with $f(0,\psi)$ led by the term $d_q^{m(Q^{\l})}t^{qm(Q^{\l})}$, hence 
$$[J_{\ell(\Omega)}^1(\Omega)]=[\mu_{m(Q^{\l})}(\mathbb C)]\Lbb^{q(m(Q^{\l})-1)}.$$ 
If $p\leq qm(Q^{\l})$, $J_{\ell(\Omega)}^1(\Omega)$ consists of pairs $(\varphi,\psi)$ with $\ord_t\varphi=p$ and $\psi$ as previous, and the leading term of $f(\varphi,\psi)$ is $d_q^{m(Q^{\l})}t^{qm(Q^{\l})}$. It follows that 
$$[J_{\ell(\Omega)}^1(\Omega)]=[\mu_{m(Q^{\l})}(\mathbb C)]\Lbb^{2qm(Q^{\l})-(p+q)}.$$ 
Thus $\sum_{\Omega\in\sigma(P_{0,1})}\Phi_1(T)$ is the sum of the series
$$[\mu_{m(Q^{\l})}(\mathbb C)]\sum_{(p,q)\in\sigma(P_{0,1}), p> qm(Q^{\l})}\Lbb^{-q(m(Q^{\l})+1)}T^{qm(Q^{\l})}$$
and
$$
[\mu_{m(Q^{\l})}(\mathbb C)]\sum_{(p,q)\in\sigma(P_{0,1}), p\leq qm(Q^{\l})}\Lbb^{-(p+q)}T^{qm(Q^{\l})}.$$
Taking $\lim_{T\to\infty}$, the former has image $-[\mu_{m(Q^{\l})}(\mathbb C)]$ by a direct computation or due to the proof of \cite[Lemma 3.3.3]{G}, the latter has image zero by \cite[Lemma 2.10]{GLM}. 

The case where $i=m_{\mathscr B_{s}}$ is proved similarly.
\end{proof}

\begin{lemma}\label{lem:3.5} 
The following identities hold:
\begin{gather*}
\lim_{T\to\infty}\sum_{\Omega\in\sigma(P_{\mathscr B_{\s},i})}\Phi_2(T)=(\Lbb-1)\sum_{j=1}^{r_{\mathscr B_{\s},i}}[\mu_{A_{\mathscr B_{\s},i,j}}(\mathbb C)], \quad 1\leq i\leq m_{\mathscr B_{\s}},\\
\lim_{T\to\infty}\sum_{\Omega\in\sigma(P_{\mathscr B_{\s},i,i+1})}\Phi_2(T)=0, \quad 0\leq i\leq m_{\mathscr B_{\s}}.
\end{gather*}
\end{lemma}

\begin{proof}
We only need to write down a proof for the first identity because the second one is trivial. For each $\Omega$ in $\sigma(P_{\mathscr B_{\s},i})\cap\mathbb Z^2$ and each $n=\ell_{\mathscr B_{\s}}(\Omega)+k$, $k\geq 1$, let us consider the projection 
$$\rho: J_{\mathscr B_{\s},n}(\Omega)\to f_{\mathscr B_{\s},i}^{-1}(0)$$ 
which sends $(\varphi(t),\psi(t))$ to their leading nonzero coefficients. Defining $J_{\mathscr B_{\s},n}^{(j)}(\Omega)$ to be 
$$\rho^{-1}\left(\left\{(x,y)\in\mathbb G_{\mathbb C}^2\mid y^{a_{\mathscr B_{\s},i}}+\xi_{\mathscr B_{\s},i,j}x^{b_{\mathscr B_{\s},i}}=0, y^{a_{\mathscr B_{\s},i}}+\xi_{\mathscr B_{\s},i,h}x^{b_{\mathscr B_{\s},i}}\not=0\ \forall h\not=j\right\}\right),$$ 
we can decompose $J_{\mathscr B_{\s},n}(\Omega)$ into a disjoint union of subsets $J_{\mathscr B_{\s},n}^{(j)}(\Omega)$ for all $1\leq j\leq r_{\mathscr B_{\s},i}$. Recall that, since $\Omega\in\sigma(P_{\mathscr B_{\s},i})\cap\mathbb Z^2$, it has the form $(\alpha a_{\mathscr B_{\s},i},\alpha b_{\mathscr B_{\s},i})$ for some $\alpha\in \mathbb N_{>0}$, and $\ell_{\mathscr B_{\s}}(\Omega)=\alpha m(P_{\mathscr B_{\s},i})$. Thus, for such an $\Omega$ and $k\geq 1$, the class $[J_{\mathscr B_{\s},\ell_{\mathscr B_{\s}}(\Omega)+k}^{(j)}(\Omega)]$ equals 
\begin{align*}
[J_{\alpha a_{\mathscr B_{\s},i}b_{\mathscr B_{\s},i}A_{\mathscr B_{\s},i,j}+k}((y^{a_{\mathscr B_{\s},i}}+\xi_{\mathscr B_{\s},i,j}x^{b_{\mathscr B_{\s},i}})^{A_{\mathscr B_{\s},i,j}})(\Omega)]\Lbb^{2\left(\ell_{\mathscr B_{\s}}(\Omega)-\alpha a_{\mathscr B_{\s},i}b_{\mathscr B_{\s},i}A_{\mathscr B_{\s},i,j}\right)}.
\end{align*}
By definition, we have
$$J_{a_{\mathscr B_{\s},i}b_{\mathscr B_{\s},i}A_{\mathscr B_{\s},i,j}+k}((y^{a_{\mathscr B_{\s},i}}+\xi_{\mathscr B_{\s},i,j}x^{b_{\mathscr B_{\s},i}})^{A_{\mathscr B_{\s},i,j}})(\Omega)=\emptyset$$ 
if $A_{\mathscr B_{\s},i,j}$ does not divide $\alpha a_{\mathscr B_{\s},i}b_{\mathscr B_{\s},i}A_{\mathscr B_{\s},i,j}+k$. Otherwise, we put 
$$\alpha a_{\mathscr B_{\s},i}b_{\mathscr B_{\s},i}A_{\mathscr B_{\s},i,j}+k=\lambda A_{\mathscr B_{\s},i,j}.$$ 
Since $y^{a_{\mathscr B_{\s},i}}+\xi_{\mathscr B_{\s},i,j}x^{b_{\mathscr B_{\s},i}}$ is non-degenerate with respect to its Newton polyhedron, we may use the argument in the proof of \cite[Lemma 2.1.1]{G} and obtain 
\begin{align*}
[J_{\lambda}(y^{a_{\mathscr B_{\s},i}}+\xi_{\mathscr B_{\s},i,j}x^{b_{\mathscr B_{\s},i}})(\Omega)]
=(\Lbb-1)\Lbb^{2\lambda-|\Omega|}.
\end{align*}
It implies that, for $n=\ell_{\mathscr B_{\s}}(\Omega)+k=\ell_{\mathscr B_{\s}}(\Omega)+(\lambda-\alpha a_{\mathscr B_{\s},i}b_{\mathscr B_{\s},i})A_{\mathscr B_{\s},i,j}$ and $\lambda>\alpha a_{\mathscr B_{\s},i}b_{\mathscr B_{\s},i}$, we have
\begin{align*}
[J_{\mathscr B_{\s},n}^{(j)}(\Omega)]=(\Lbb-1)[\mu_{A_{\mathscr B_{\s},i,j}}(\mathbb C)]\Lbb^{2\left(\ell_{\mathscr B_{\s}}(\Omega)-\alpha a_{\mathscr B_{\s},i}b_{\mathscr B_{\s},i}A_{\mathscr B_{\s},i,j}\right)+\left(2\lambda-|\Omega|\right)A_{\mathscr B_{\s},i,j}}.
\end{align*}
Then 
\begin{align*}
&\sum_{\Omega\in\sigma(P_{\mathscr B_{\s},i})}\sum_{n>\ell_{\mathscr B_{\s}}(\Omega)}[J_{\mathscr B_{\s},n}^{(j)}(\Omega)]\Lbb^{-2n}T^n\\
&=(\Lbb-1)[\mu_{A_{\mathscr B_{\s},i,j}}(\mathbb C)]\sum_{\Omega\in\sigma(P_{\mathscr B_{\s},i})}\Lbb^{-|\Omega|A_{\mathscr B_{\s},i,j}}T^{\ell_{\mathscr B_{\s}}(\Omega)}\sum_{\lambda>a_{\mathscr B_{\s},i}b_{\mathscr B_{\s},i}}T^{(\lambda-a_{\mathscr B_{\s},i}b_{\mathscr B_{\s},i})A_{\mathscr B_{\s},i,j}}\\
&=(\Lbb-1)[\mu_{A_{\mathscr B_{\s},i,j}}(\mathbb C)]\frac{T^{A_{\mathscr B_{\s},i,j}}}{1-T^{A_{\mathscr B_{\s},i,j}}}\sum_{\Omega\in\sigma(P_{\mathscr B_{\s},i})}\Lbb^{-|\Omega|A_{\mathscr B_{\s},i,j}}T^{\ell_{\mathscr B_{\s}}(\Omega)}.
\end{align*}
Now it follows from \cite[Lemma 2.1.5]{G} that 
$$\lim_{T\to\infty}\sum_{\Omega\in\sigma(P_{\mathscr B_{\s},i})}\sum_{n>\ell_{\mathscr B_{\s}}(\Omega)}[J_{\mathscr B_{\s},n}^{(j)}(\Omega)]\Lbb^{-2n}T^n=(\Lbb-1)[\mu_{A_{\mathscr B_{\s},i,j}}(\mathbb C)],$$
which proves the lemma.
\end{proof}


\subsection{Description of $\mathscr S_{f,O}$ via $\mathbf G_{\s}$}
Let $\mathbf B_{\s}$ denote the set of bamboos of $\mathbf G_{\s}$. Let us recall the face functions corresponding to each bamboo $\mathscr B$ in $\mathbf B_{\s}$ as in Subsection \ref{subsec:3.4}, up to factor in $\mathbb G_{m,\mathbb C}$:
$$f_{P_{\mathscr B,i}}(u,v)=v^{\sum_{t=i+1}^{m_{\mathscr B}}a_{\mathscr B,t}A_{\mathscr B,t}}u^{m(P[\mathscr B])+\sum_{t=1}^{i-1}b_{\mathscr B,t}A_{\mathscr B,t}}\prod_{j=1}^{r_{\mathscr B,i}}(v^{a_{\mathscr B,i}}+\xi_{\mathscr B,i,j}u^{b_{\mathscr B,i}})^{A_{\mathscr B,i,j}}$$
for $1\leq i\leq m_{\mathscr B}$, and
$$f_{P_{\mathscr B,i,i+1}}(u,v)=v^{\sum_{t=i+1}^{m_{\mathscr B}}a_{\mathscr B,t}A_{\mathscr B,t}}u^{m(P[\mathscr B])+\sum_{t=1}^ib_{\mathscr B,t}A_{\mathscr B,t}}$$
for $0\leq i\leq m_{\mathscr B}$. Here $m(P[\mathscr B])$ is the multiplicity of $E(P[\mathscr B])$, the predecessor of $E(P_{\mathscr B,1})$ in $\mathbf G_{\s}$, $m(P[\mathscr B])=0$ for $\mathscr B$ being the first bamboo of $\mathbf G_{\s}$, and, by convention, $\sum_{t=1}^0=0$, $\sum_{t=m_{\mathscr B}+1}^{m_{\mathscr B}}=0$. To each $\mathscr B\in \mathbf B_{\s}$ we associate varieties 
$$
X(\mathscr B,i):=\{(u,v)\in\mathbb A_{\mathbb C}^2\mid f_{P_{\mathscr B,i}}(u,v)=1\},$$
for $1\leq i\leq m_{\mathscr B}$, and 
$$
X(\mathscr B,i,i+1):=\{(u,v)\in\mathbb A_{\mathbb C}^2\mid f_{P_{\mathscr B,i,i+1}}(u,v)=1\},$$
for $1\leq i< m_{\mathscr B}$.

\begin{theorem}\label{main1}
The motivic Milnor fiber $\mathscr S_{f,O}$ is expressed via $\mathbf G_{\s}$ as follows
\begin{align*}
\mathscr S_{f,O}=\sum_{\mathscr B\in\mathbf B_{\s}}\left(\sum_{i=1}^{m_{\mathscr B}}[X(\mathscr B,i)]-\sum_{i=1}^{m_{\mathscr B}-1}[X(\mathscr B,i,i+1)]-(\Lbb-1)\sum_{i=1}^{m_{\mathscr B}}\sum_{j=1}^{r_{\mathscr B,i}}[\mu_{A_{\mathscr B,i,j}}(\mathbb C)]\right).
\end{align*}
\end{theorem}

\begin{proof}
For any bamboo $\mathscr B$ of $\mathbf G_{\s}$ and $1<i<m_{\mathscr B}$, the complex variety $X(\mathscr B,i)$ is nothing else than $\{(u,v)\in\mathbb G_{m,\mathbb C}^2\mid f_{P_{\mathscr B,i}}(u,v)=1\}$, which is exactly $f_{P_{\mathscr B,i}}^{-1}(1)$ as denoted in Proposition \ref{prop:3.3}. In these notation we have
$$[X(\mathscr B_1,1)]=[f_{P_{\mathscr B_1,1}}^{-1}(1)]+[\mu_{m(Q^{\l})}(\mathbb C)]$$
and
$$[X(\mathscr B,m_{\mathscr B})]=[f_{P_{\mathscr B,m_{\mathscr B}}}^{-1}(1)]+[\mu_{m(Q_{\mathscr B}^{\r})}(\mathbb C)].$$
The theorem now follows directly from Propositions \ref{prop:3.2} and \ref{prop:3.3}. 
\end{proof}

\begin{examples}
{\it a)} Let $f(x,y)=(y^a+x^b)^A$, where $a$, $b$, $A$ are in $\mathbb N_{>0}$ and $(a,b)=1$. Applying Theorem \ref{main1} we get
$$\mathscr S_{f,O}=[\{(x,y)\in\mathbb A_{\mathbb C}^2\mid (y^a+x^b)^A=1\}]-(\Lbb-1)[\mu_A(\mathbb C)].$$

\noindent {\it b)} Let $(f,O)$ be an irreducible singularity with weight vectors $(a_i,b_i)^t$, $1\leq i\leq g$. The graph $\mathbf G_{\s}$ is obtained in terms of a sequence of toric modifications with respect to these weight vectors. 

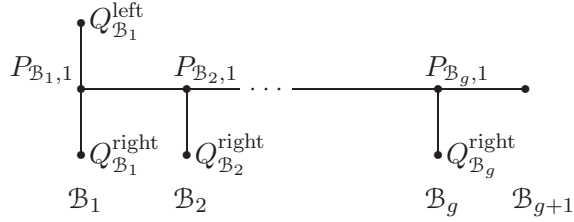
\begin{center}
\begin{picture}(160,95)(-30,-55)
\put(-40,0){\line(1,0){60}}
\put(25,0){\circle*{1}}
\put(30,0){\circle*{1}}
\put(35,0){\circle*{1}}
\put(40,0){\line(1,0){87}}
\put(-40,0){\circle*{3}}
\put(-67,5){$P_{\mathscr B_1,1}$}
\put(-40,0){\line(0,1){25}}
\put(-40,0){\line(0,-1){25}}
\put(-45,-45){$\mathscr B_1$}
\put(-40,-25){\circle*{3}}
\put(-37,-28){$Q_{\mathscr B_1}^{\r}$}
\put(-37,23){$Q_{\mathscr B_1}^{\l}$}
\put(-40,25){\circle*{3}}
\put(0,0){\circle*{3}}
\put(-5,5){$P_{\mathscr B_2,1}$}
\put(0,0){\line(0,-1){25}}
\put(0,-25){\circle*{3}}
\put(3,-28){$Q_{\mathscr B_2}^{\r}$}
\put(-5,-45){$\mathscr B_2$}
\put(95,0){\circle*{3}}
\put(95,0){\line(0,-1){25}}
\put(95,-25){\circle*{3}}
\put(90,5){$P_{\mathscr B_g,1}$}
\put(98,-28){$Q_{\mathscr B_g}^{\r}$}
\put(90,-45){$\mathscr B_g$}
\put(128,0){\circle*{3}}
\put(123,-45){$\mathscr B_{g+1}$}
\end{picture}
\end{center}
Let $\mathscr B_i,\dots, \mathscr B_{g+1}$ denote the bamboos of $\mathbf G_{\s}$, where $\mathscr B_i$ is the unique bamboo of the $i$th floor. Note that, in this case, $P_{\mathscr B_i,1}=(a_i,b_i)^t$, $A_{\mathscr B_i,1,1,1}=a_{i+1}\cdots a_g$, and the multiplicities of the resolution on the exceptional divisors $E(P_{\mathscr B_i,1})$ are computed as follows:
\begin{gather*}
m(P_{\mathscr B_0,1})=0 \quad\text{convention},\\
m(P_{\mathscr B_1,1})=a_1b_1A_{\mathscr B_1,1,1,1}=a_1\cdots a_gb_1,\\
m(P_{\mathscr B_i,1})=a_im(P_{\mathscr B_{i-1},1})+a_ib_iA_{\mathscr B_i,1,1,1},\quad i\geq 2.
\end{gather*}
Then by Theorem \ref{main1} we have
$$\mathscr S_{f,O}=\sum_{i=1}^g\left([\{(x,y)\in\mathbb A_{\mathbb C}^2\mid x^{m(P_{\mathscr B_{i-1},1})}(y^{a_i}+x^{b_i})^{A_{\mathscr B_i,1,1,1}}=1\}]-(\Lbb-1)[\mu_{A_{\mathscr B_i,1,1,1}}(\mathbb C)]\right).$$
\end{examples}

\begin{remark}\label{rk46}
It is a fact that Hodge-Steenbrink spectrum is a crucial invariant in singularity theory. By \cite{St}, $H^j(M_f,\mathbb C)$ carries a canonical mixed Hodge structure compatible with the semisimple part of the monodromy $T_f$. This gives rise to the Hodge-Steenbrink spectrum $\Sp(f,O)$ of the singularity $f$, which is a fractional Laurent polynomial $\sum_{\alpha\in \mathbb Q}n_{\alpha}(f)t^{\alpha}$, where $n_{\alpha}(f)=\sum_{j\in \mathbb Z}(-1)^j\dim_{\mathbb C}Gr_F^{\lfloor{2-\alpha}\rfloor}H^{1+j}(M_f,\mathbb C)_{e^{-2\pi i\alpha}}$, with $F^{\bullet}$ the Hodge filtration. Using a Hodge realization Denef-Loeser \cite{DL} construct a linear map $\Sp: \mathscr M_{\mathbb C}^{\hat{\mu}}\to \mathbb Z[\mathbb Q]$, which is a ring homomorphism with respect to the convolution product $\ast$ in $\mathscr M_{\mathbb C}^{\hat{\mu}}$ (see \cite{GLM}), and by that work we have $\Sp(f,O)=\Sp(\mathscr S_{f,O})$. Now, it follows from \cite[Lemme 3.4.2]{G}, \cite[(2.1.2)]{Sa2} and Theorem \ref{main1} that 
$$\Sp(f,O)=\sum_{\mathscr B\in\mathbf B_{\s}}\left(\sum_{i=1}^{m_{\mathscr B}}\Sp([X(\mathscr B,i)])-\sum_{i=1}^{m_{\mathscr B}-1}\Sp([X(\mathscr B,i,i+1)])+\sum_{i=1}^{m_{\mathscr B}}\sum_{j=1}^{r_{\mathscr B,i}}\frac{(t-1)^2}{1-t^{1/A_{\mathscr B,i,j}}}\right).$$
Therefore in order to compute $\Sp(f,O)$ it suffices to study the spectrum of a quasi-homogeneous plane curve singularity.
\end{remark}

\begin{ack}
The author would like to thank The Abdus Salam International Centre for Theoretical Physics (ICTP), The Vietnam Institute for Advanced Study in Mathematics (VIASM) and Department of Mathematics - KU Leuven for warm hospitality during his visits. 
\end{ack}

\end{document}